\documentclass[12pt]{article}
\usepackage{amssymb,latexsym,amsmath}
\usepackage{amsthm}
\usepackage{graphicx,color}
\usepackage{geometry}
\newtheorem{thm}{Theorem}

\theoremstyle{plain}

\theoremstyle{definition}
\newtheorem{dfn}{Definition}

\numberwithin {equation}{section}

\usepackage{longtable}

\begin{document}
\title
{An infinite dimensional pursuit--evasion differential game with
finite number of players}
\author{Mehdi Salimi$^a$\thanks{Corresponding author: mehdi.salimi@tu-dresden.de, mehdi.salimi@medalics.org}
\and Massimiliano Ferrara$^b$\thanks{massimiliano.ferrara@unirc.it}}
\date{}
\maketitle
\begin{center}
$^{a}$ Center for Dynamics, Department of Mathematics, Technische
Universit\"{a}t Dresden, Germany\\
$^{b}$ Department of Law and Economics, Universit\`{a}
Mediterranea di Reggio Calabria, Italy\\
$^{a,b}$ MEDAlics, Research Center at Universit\`{a} per Stranieri
Dante Alighieri, Reggio Calabria, Italy
\end{center}
\date{}
\maketitle
\begin{abstract}
We study a pursuit--evasion differential game with finite number
of pursuers and one evader in Hilbert space with geometric
constraints on the control functions of players. We solve the game
by presenting explicit strategies for pursuers which guarantee
their pursuit as well as an strategy for the evader which
guarantees its evasion.

 \end{abstract}
\section{Introduction and preliminaries}

Differential games and pursuit--evasion problems are investigated
by many authors and significant researches are given by Isaacs
\cite{isa} and Petrosyan \cite{pet}.

Ibragimov and Salimi \cite{is} study a differential game of
optimal approach of countably many pursuers to one evader in an
infinite dimensional Hilbert space with integral constraints on
the controls of the players. Ibragimov et al.\ \cite{isaa} study
an evasion problem from many pursuers in a simple motion
differential game with integral constraints. In \cite{sal} Salimi
et al.\ investigate a differential game in which countably many
dynamical objects pursue a single one. All the players perform
simple motions. The duration of the game is fixed. The controls of
a group of pursuers are subject to integral constraints and the
controls of the other pursuers and the evader are subject to
geometric constraints. The payoff of the game is the distance
between the evader and the closest pursuer when the game is
terminated. They construct optimal strategies for players and find
the value of the game.

In the present paper, we solve a pursuit--evasion differential
game with geometric constraints on the controls of players. In
other words a pursuit of one player by finite number of dynamical
players.

In the Hilbert space $\ell_2 = \{\alpha = (\alpha_k)_{k \in
\mathbf{N}} \in \mathbf{R}^{\mathbf{N}} :
\sum_{k=1}^\infty\alpha_k^2<\infty\}$ with inner product
$(\alpha,\beta)=\sum_{k=1}^\infty\alpha_k\beta_k$, the motions of
the pursuers $P_i$ and the evader $E$ are defined by the
equations:

\begin{equation}\label{ab0}
 \begin{split}
&(P_i):\dot{x}_i=u_i(t), \quad x_i(0)=x_{i0},
, \quad i=1,2,\ldots,m\\
&(E):\dot{y}=v(t), \quad y(0)=y_0,
 \end{split}
\end{equation}
where $x_i,x_{i0},y,y_0\in \ell_2,$
$u_i=(u_{i1},u_{i2},\ldots,u_{i\zeta},\ldots)$ is the control
parameter of the pursuer $P_i$, and
$v=(v_1,v_2,\ldots,v_\zeta,\ldots)$ is that of the evader $E$. In
the following definitions $i=1,2,\ldots,m$.
\begin{dfn}
A function $u_i(\cdot),$ $u_i:[0,\infty)\to \ell_2,$ such that
$u_{i\zeta}:[0,\infty)\to R^1,$ $\zeta=1,2,\ldots,$ are Borel
measurable functions and
$$
\|u_i(t)\|=\Big( \sum_{\zeta=1}^{\infty} u_{i\zeta}(t)^{2} \,dt
\Big)^{\frac{1}{2}}
  \leq
  1
$$
is called an \textit{admissible control of the ith pursuer}.
\end{dfn}
\begin{dfn}
A function $v(\cdot),$ $v:[0,\infty)\to \ell_2,$ such that
$v_\zeta:[0,\infty)\to R^1,$ ${\zeta=1,2,\ldots,}$ are Borel
measurable functions and
$$
\|v(t)\|=\Big(\sum_{\zeta=1}^{\infty} v_{\zeta}(t)^{2} \,dt
\Big)^{\frac{1}{2}}
  \leq
  1,
$$
is called an \textit{admissible control of the evader}.
\end{dfn}
Once the players' admissible controls $u_i(\cdot)$ and $v(\cdot)$
are chosen, the corresponding motions $x_i(\cdot)$ and $y(\cdot)$
of the players are defined as
$$
x_i(t)=(x_{i1}(t),x_{i2}(t),\ldots,x_{i\zeta}(t),\ldots), \quad
y(t)=(y_1(t),y_2(t),\ldots,y_\zeta(t),\ldots),
$$
\begin{equation*}
x_{i\zeta}(t)=x_{i\zeta0}+\int \limits_0^tu_{i\zeta}(s)\,ds, \quad
y_\zeta(t)=y_{\zeta0}+\int\limits_0^tv_\zeta(s)\,ds.
\end{equation*}
\begin{dfn}
A function $U_i(t,x_i,y,v),$ $U_i:[0,\infty)\times \ell_2\times
\ell_2\times \ell_2\to \ell_2,$ such that the system
\begin{equation*}
 \begin{split}
&\dot{x}_i=U_i(t,x_i,y,v), \quad x_i(0)=x_{i0},\\
&\dot{y}=v, \quad y(0)=y_0,
 \end{split}
\end{equation*}
has a unique solution $(x_i(\cdot),y(\cdot))$ for an arbitrary
admissible control $v=v(t),$ $0\le t<\infty,$ of the evader $E,$
is called a \textit{strategy of the pursuer~$P_i$}. A strategy
$U_i$ is said to be \textit{admissible} if each control formed by
this strategy is admissible.
\end{dfn}
\begin{dfn}
A function $V(t,x_1,\ldots,x_m,y),$ $V:[0,\infty)\times
\underbrace{\ell_2\times\ldots\times \ell_2}_{m+1}\to \ell_2,$
such that the system of equations
\begin{equation*}
 \begin{split}
&\dot{x}_i=u_i, \quad x_i(0)=x_{i0},\\
&\dot{y}=V(t,x_1,\ldots,x_m,y), \quad y(0)=y_0,
 \end{split}
\end{equation*}
has a unique solution $(x_1(\cdot),\ldots,x_m(\cdot),y(\cdot))$
for arbitrary admissible controls $u_i=u_i(t),$ $0\le t<\infty,$
of the pursuers $P_i,$ is called a \textit{strategy} of the evader
$E.$ If each control formed by a strategy $V$ is admissible, then
the strategy $V$ itself is said to be \textit{admissible}.
\end{dfn}
\section{Pursuit problem and its solution}

\begin{dfn}
If $x_i(\tau)=y(\tau)$ at some $i$ and $\tau>0$, then pursuit is
considered complete.
\end{dfn}
\begin{thm}
Suppose the initial positions of the pursuers and the evader in
the game (\ref{ab0}) are different and for any non-zero vector
$p\in \ell_2$, there is $k\in \{1,2,\ldots,m\}$ such that
$(y_0-x_{k0},p)<0$, then pursuit is complete.
\end{thm}
\begin{proof}
We define the pursuers' strategy as follow:
\begin{equation}\label{a}
u_i(t)=v(t)-\left(v(t),e_i\right)e_i+e_i\left(1-\|v(t)\|^2+\left(v(t),
e_i\right)^2\right)^{1/2},
\end{equation}
where $e_i=\dfrac{y_0-x_{i0}}{\|y_0-x_{i0}\|}$, $i=1,2,\ldots,m$.

The above strategy is admissible. Indeed

\begin{equation*}
\begin{split}
\|u_i(t)\|^2&=\|v(t)-(v(t),~e_i)e_i\|^2+2\big(v(t)-(v(t),~e_i)e_i,~e_i\big(1-\|v(t)\|^2+(v(t),~e_i)^2\big)^{1/2}\big)\\
&+1-\|v(t)\|^2+(v(t),~e_i)^2\\
&=\|v(t)\|^2-2(v(t),~e_i)^2+(v(t),~e_i)^2+(v(t),~e_i)^2+1-\|v(t)\|^2\leq
1.
\end{split}
\end{equation*}

By (\ref{a}), we have $y(t)-x_i(t)=e_i \Omega_i(t)$, where
\begin{equation*}
\Omega_i(t)=\|y_0-x_{i0}\|-\int_0^t\left(\left(1-\|v(s)\|^2+\left(v(s),e_i\right)^2\right)^{1/2}-\left(v(s),e_i\right)\right)
\,ds.
\end{equation*}
We are going to show that $\Omega_i(\tau)=0$, for some $i=1,2,\ldots,m$, and $\tau>0$.\\

It is clear that $\Omega_i(0)=\|y_0-x_{i0}\|>0$ for
$i=1,2,\ldots,m$.

Suppose that $\Omega(t)=\sum_{i=1}^{m}\Omega_i(t)$, thus
\begin{equation*}
\Omega(t)=\sum_{i=1}^{m}\|y_0-x_{i0}\|-\int_{0}^{t}\sum_{i=1}^{m}\left(\left(1-\|v(s)\|^2+(v(s),e_i)^2\right)^{1/2}-(v(s),e_i)\right)\,ds.
\end{equation*}
Obviously
\begin{equation*}
\Lambda(v)=\sum_{i=1}^{m}\left(\left(1-\|v\|^2+(v,e_i)^2\right)^{1/2}-(v,e_i)\right)\geq0.
\end{equation*}
Define
\begin{equation*}
\Theta:=\inf_{\|v\|\leq1}\Lambda(v),
\end{equation*}
so $\Theta>0 ~~\text{or}~~ \Theta=0$.

We show that $\Theta\neq0$. Assume by contradiction that
$\Theta=0$. Then there exists a minimizing sequence
$\{v_n\}_n\subset \ell_2$ with $\|v_n\|\leq 1$ for the value
$\Theta=0$, i.e.,
\begin{equation}\label{elso-osszef}
\lim_{n\to \infty}\Lambda(v_n)=0,
\end{equation}
 where
$$\Lambda(v)=\sum_{i=1}^m\left((1-\|v\|^2+(v,e_i)^2)^\frac{1}{2}-(v,e_i)\right)\geq 0.$$

On one hand, since the unit ball $B=\{v\in \ell_2:\|v\|\leq 1\}$
is weakly compact (but not strongly compact due to the fact that
$\ell_2$ is an infinite dimensional Hilbert space), we may extract
a subsequence (denoted in the same way) from $\{v_n\}$ which
converges weakly to an element $v_0\in B$, i.e.
$$v_n \xrightarrow[]{*} v_0\ {\rm as}\ n \xrightarrow[]{} \infty.$$
In particular, this fact implies that
\begin{equation}\label{masodik-osszef}
\lim_{n\to \infty}(v_n,w)= (v_0,w),\ \forall w\in \ell_2.
\end{equation}

On the other hand, since the real-valued sequence $\{\|v_n\|\}_n$
is bounded, up to some subsequence, we may assume that it
converges to an element $c_0\in [0,1]$. We claim that $c_0=1$.
Assume that $c_0<1.$ Then, we have by (\ref{elso-osszef}) and
(\ref{masodik-osszef}) that
\begin{eqnarray*}
  0 &=& \lim_{n\to \infty}\Lambda(v_n)\\&=&\lim_{n\to
\infty}\sum_{i=1}^m\left((1-\|v_n\|^2+(v_n,e_i)^2)^\frac{1}{2}-(v_n,e_i)\right) \\
  &=&
  \sum_{i=1}^m\left((1-c_0^2+(v_0,e_i)^2)^\frac{1}{2}-(v_0,e_i)\right)\\&>&0,
\end{eqnarray*}
a contradiction. Therefore, $c_0=1,$ i.e., $\lim_{n\to
\infty}\|v_n\|=1.$

Now, we come back again to (\ref{elso-osszef}), obtaining that
\begin{eqnarray*}
  0 &=& \lim_{n\to \infty}\Lambda(v_n)\\&=&\lim_{n\to
\infty}\sum_{i=1}^m\left((1-\|v_n\|^2+(v_n,e_i)^2)^\frac{1}{2}-(v_n,e_i)\right) \\
  &=&
  \sum_{i=1}^m\left(((v_0,e_i)^2)^\frac{1}{2}-(v_0,e_i)\right)\\&=&\sum_{i=1}^m\left(|(v_0,e_i)|-(v_0,e_i)\right).
\end{eqnarray*}
Since every term under the sum is non-negative, we necessarily
have that $|(v_0,e_i)|=(v_0,e_i)$ for all $i\in \{1,...,m\}.$
Therefore,
$$(v_0,e_i)=|(v_0,e_i)|\geq 0,\ \forall i\in \{1,...,m\},$$ which is inconsistent with the hypothesis of the theorem, therefore $\Theta>0$.

So,
\begin{equation*}
\Omega(t)\leq \Omega(0)-\int_{0}^{t}\Theta \,ds=\Omega(0)-\Theta
t,
\end{equation*}
therefore, in time $\eta=\frac{\Omega(0)}{\Theta}$ we have
$\Omega(\eta)=\sum_{i=1}^{m}\Omega_i(\eta)\leq0$, and then
$\Omega_i(\tau)=0$, for some $i=1,2,\ldots,m$, $\tau \in(0,\eta]$
and pursuit is complete.
\end{proof}
\section{Evasion problem and its solution}

\begin{dfn}
If there exists a strategy of the evader such that $x_i(t)\neq
y(t)$, $t>0$, then evasion is possible.
\end{dfn}
\begin{thm}
Suppose the initial positions of the pursuers and the evader in
the game (\ref{ab0}) are different and there exists a non-zero
vector $p\in \ell_2$ such that $\|p\|=1$ and
$(y_0-x_{i0},p)\geq0$, $i\in \{1,2,\ldots,m\}$, then evasion is
possible.
\end{thm}
\begin{proof}
We define the evader's strategy as follow:
\begin{equation}\label{ae}
v(t)=p, \quad t\geq 0.
\end{equation}
Obviously the above strategy is admissible.

We have
\begin{equation*}
\left(y(t)-x_i(t),p\right)=\left(y_0-x_{i0},p\right)+\int_{0}^{t}\left(v(s),p\right)\,ds-\int_{0}^{t}\left(u_i(s),p\right)\,ds.
\end{equation*}

By taking strategy (\ref{ae}) we obtain

\begin{equation*}
\left(y(t)-x_i(t),p\right)=\left(y_0-x_{i0},p\right)+\int_{0}^{t}\left[1-\left(u_i(s),p\right)\right]\,ds.
\end{equation*}

Let's assume that evasion is not possible, so there are $\tau
>0$ and $k\in \{1,2,\ldots,m\}$ such that $y(\tau)=x_k(\tau)$.
Then

\begin{equation*}
\left(y(\tau)-x_k(\tau),p\right)=(y_0-x_{k0},p)+\int_{0}^{\tau}\left[1-\left(u_k(s),p\right)\right]\,ds=0.
\end{equation*}

By the assumption of the theorem, $(y_0-x_{k0},p)\geq0$. On the
other hand $\|u_i(t)\|\leq 1$, and then

\begin{equation*}
|(u_i(t),p)|\leq \|u_i(t)\|\cdot\|p\|\leq 1, \quad
t\in\left[0,\tau\right],
\end{equation*}

so $(u_i(t),p)\leq 1$. Thus $(y_0-x_{k0},p)=0$ and then

\begin{equation*}
\int_{0}^{\tau}\left[1-\left(u_k(s),p\right)\right]\,ds=0.
\end{equation*}

From the above equality we obtain $1-\left(u_k(s),p\right)=0$,
$s\in\left[0,\tau\right]$, almost everywhere. Hence
$\left(u_k(s),p\right)=1$, and $u_k(s)=p$,
$s\in\left[0,\tau\right]$.

Therefore

\begin{equation*}
y(\tau)-x_k(\tau)=y_0+\int_{0}^{\tau}p
\,ds-x_{k0}-\int_{0}^{\tau}p \,ds=y_0-x_{k0}=0,
\end{equation*}

which is a contradiction with the initial positions of the
pursuers and the evader. So $x_i(t)\neq y(t)$, $i\in
\{1,2,\ldots,m\}$, $t>0$. In other words, evasion of the evader
from all the pursuers is possible.
\end{proof}
\section{Conclusion}

We considered a pursuit--evasion problem with finite number of
pursuers and one evader in the Hilbert space $\ell_2$. The
controls of pursuers and the evader are subject to geometric
constraints. We constructed admissible strategies for pursuers
which guarantee capture of the evader as well as an admissible
strategy for the evader which guarantees evasion from all
pursuers.
\\\\
 {\bf Acknowledgments.} The present research was supported by the MEDAlics, Research Center at Universit\`{a} per Stranieri Dante Alighieri, Reggio Calabria, Italy.

\end{document}